\documentclass[letterpaper, 10 pt, conference]{ieeeconf}  % Comment this line out if you need a4paper

\IEEEoverridecommandlockouts                              % This command is only needed if 

\usepackage{cite}
\newlength{\bibitemsep}
\newlength{\bibparskip}
\let\oldthebibliography\thebibliography
\renewcommand\thebibliography[1]{%
  \oldthebibliography{#1}%
  \setlength{\parskip}{\bibitemsep}%
  \setlength{\itemsep}{\bibparskip}%
}
\renewcommand{\epsilon}{\varepsilon}
\usepackage[utf8]{inputenc} % allow utf-8 input
\usepackage[T1]{fontenc}    % use 8-bit T1 fonts
\usepackage{hyperref}       % hyperlinks
\usepackage{url}            % simple URL typesetting
\usepackage{booktabs}       % professional-quality tables
\usepackage{amsfonts}       % blackboard math symbols
\usepackage{nicefrac}       % compact symbols for 1/2, etc.
\usepackage{microtype}      % microtypography
\usepackage{xcolor}         % colors
\usepackage{dsfont,amssymb,amsmath,subfig, graphicx,fancyhdr,mdframed}
\usepackage{dsfont,mathtools, mathrsfs,wrapfig} 
\let\labelindent\relax
\usepackage{enumitem}

\usepackage{algorithm, algorithmic}
\usepackage{cleveref}
\newtheorem{theorem}{Theorem}[section]
\newtheorem{definition}[theorem]{Definition}
\newtheorem{lemma}[theorem]{Lemma}

\newcommand{\R}{\mathbb{R}}

\DeclareMathOperator{\dom}{dom}
\newcommand{\la}{\lambda}
\newcommand{\x}{\bar x}
\newcommand{\n}[1]{\|#1 \|} 
\newcommand{\lr}[1]{ \langle #1 \rangle}
\newcommand{\set}[1]{\mathcal{#1}}
\newcommand{\PP}{\mathds{P}}
\usepackage{comment}
\newenvironment{customproof}[1][Proof]%
   {\par\noindent\textbf{#1. }\ignorespaces}
   {\hfill$\square$\par}

\allowdisplaybreaks

\title{\LARGE \bf A Hybrid Algorithm for Monotone Variational Inequalities}

\author{Reza Rahimi Baghbadorani\textsuperscript{a}, Peyman Mohajerin Esfahani\textsuperscript{a,b}, Sergio Grammatico\textsuperscript{a}  %<-this % stops a space
\thanks{The authors are with (a) Delft University of Technology and (b) University of Toronto. E-mail addresses:
{\tt\small \{r.rahimibaghbadorani, p.mohajerinesfahani, s.grammatico\}@tudelft.nl}. This work was supported by the ERC grant TRUST-949796 and the NSERC Discovery grant RGPIN-2025-06544.}%
}

\begin{document}

\maketitle
\thispagestyle{empty}
\pagestyle{empty}
%%%%%%%%%%%%%%%%%%%%%%%%%%%%%%%%%%%%%%%%%%%%%%%%%%%%%%%%%%%%%%%%%%%%%%%%%%%%%%%%
\begin{abstract}
Inspired by the adaptive Golden Ratio Algorithm (aGRAAL), we propose two new methods for solving monotone variational inequalities. We show that by selecting the momentum parameter beyond the golden ratio in aGRAAL, the convergence speed can be improved, which motivates us to study the switching between small and large momentum parameters to accelerate convergence. We validate the performance of our proposed algorithms on several classes of variational inequality problems studied in the machine learning and control literature, including Nash equilibrium seeking, composite minimization, Markov decision processes, and zero-sum games, and compare them to that of existing methods.
\end{abstract}
%%%%%%%%%%%%%%%%%%%%%%%%%%%%%%%%%%%%%%%%%%%%%%
\section{Introduction}\label{intro}
%{The variational inequality (VI) problem has recently emerged from several multi-agent control and machine learning problems \cite{benenati2024linear}, e.g., in generative adversarial networks, robust optimization, and optimal control \cite{de2011optimal}.} 
{The variational inequality (VI) problem has recently emerged from several multi-agent control and machine learning problems \cite{benenati2024linear}, e.g., in generative adversarial networks, robust optimization, and optimal control \cite{de2011optimal}.} \\ %This broad applicability of VI across diverse domains calls for efficient algorithms capable of fast solution.\\
{\textbf{Motivating example (Linear-Quadratic Dynamic Game~\cite{benenati2024linear}).}} Consider a finite-horizon open-loop linear--quadratic (LQ) dynamic game with $N$ agents and shared system state $x_t \in \mathbb{R}^n$ evolving as
\begin{align*}
x_{t+1} = A_t x_t + \sum_{i=1}^N B_t^{(i)} u_t^i, \quad t \in \{0,\dots,T-1\},
\end{align*}
where $u_t^i \in \mathbb{R}^{m_i}$ is the control input of agent $i$ at time $t$. Each agent $i$ seeks to minimize the quadratic cost
\begin{align*}
J_i(u^1,\dots,u^N) &= \ \frac{1}{2} \sum_{t=0}^{T-1} \Big( x_t^\top Q_t^{(i)} x_t + u_t^{i\top} R_t^{(i)} u_t^i 
 \\
 &+ \sum_{j\neq i} u_t^{j\top} S_t^{(i,j)} u_t^i \Big) + \frac{1}{2} x_T^\top Q_T^{(i)} x_T,
\end{align*}
with $Q_t^{(i)}, Q_T^{(i)} \succeq 0$ weighting the state deviation, $R_t^{(i)} \succ 0$ penalizing the individual control effort, and $S_t^{(i,j)}$ capturing the \emph{interaction or coupling} between the control actions of agents $i$ and $j$. We impose \emph{local constraints} on each agent's control input $u^i \in \mathcal{V}^i$ and \emph{global constraints} on the state trajectory $x_t \in \mathcal{X}$. By recursively substituting the dynamics, the state trajectory can be expressed as an affine function of the stacked control vector $u = \mathrm{col}(u^1,\dots,u^N)$. The resulting Nash equilibrium conditions can then be formulated as an affine variational inequality:
\begin{align*}
\text{find } u^* \in \mathcal{V} &:= \mathcal{V}^1 \times \dots \times \mathcal{V}^N \text{ such that } \\
\langle F(u^*), &v - u^* \rangle \ge 0, \quad \forall v \in \mathcal{V},
\end{align*}
where $F$ is the pseudogradient operator of cost $J_i$ defined as
\begin{align*}
F(u) = \begin{bmatrix} \nabla_{u^1} J_1(u) \\ \vdots \\ \nabla_{u^N} J_N(u) \end{bmatrix}.
\end{align*}
This framework illustrates how shared dynamics and interactions among agents naturally lead to a VI problem, providing a tractable benchmark for developing fast and scalable algorithms for constrained multi-agent dynamic games in various applications, such as distributed automatic generation control, vehicle platooning, and the control of autonomous vehicles navigating a crossroad \cite{benenati2024linear,baghbadorani2025douglas}.\\
This motivates fast algorithms for the following VI problem:
%In this paper, we consider the following VI problem:
\begin{equation}\label{VI-main}
\text{find} \,\, x^* \in \mathcal{V} \, \text{s.t.} \, \inf\limits_{x\in \mathcal{V}}\langle F(x^*), x - x^* \rangle + g(x) - g(x^*) \geq 0,
\end{equation}
where $\mathcal{V}$ is a finite-dimensional vector space. We assume that the operator $F$ is continuous, monotone, Lipschitz continuous, the solution set of \eqref{VI-main} is nonempty, and $g(x)$ is a proper lower semicontinuous (lsc) convex function. Problem \eqref{VI-main} can be written more traditionally as follows:
\begin{equation}\label{VI-2}
\text{find} \,\, x^* \in \mathcal{A} \quad  \text{s.t.} \quad \inf\limits_{x\in \mathcal{A}} \langle F(x^*), x - x^* \rangle \geq 0,
\end{equation}
where $g$ in \eqref{VI-main} would be the indicator function of the set $\mathcal{A}$ in \eqref{VI-2}. Note that VI problem in \eqref{VI-main} can be considered as a general form of problems in convex optimization. As an example, consider the composite minimization problem $\min_{x \in \mathbb{R}^n} f(x)+g(x)$,
%\begin{align*}
%    \min_{x \in \mathbb{R}^n} f(x)+g(x)
%\end{align*}
where $f$ is a convex and smooth function and $g$ is a proper lsc convex (and possibly nonsmooth) function. Via the KKT conditions, this problem can be written as \eqref{VI-main} with $F = \nabla f$ and the same $g$ in \eqref{VI-main} \cite{facchinei1998regularity}. Another common problem in optimization and control theory is the min-max problem. For example, consider the convex-concave saddle point problem $\min_{y \in \mathbb{R}^n} \max_{z \in \mathbb{R}^m} g_1(y) + f(y,z) - g_2(z)$,
%\begin{align*}
%    \min_{y \in \mathbb{R}^n} \max_{z \in \mathbb{R}^m} g_1(y) + f(y,z) - g_2(z)
%\end{align*}
where $g_1$ and $g_2$ are proper lsc convex functions and $f(y,z)$ is a smooth convex-concave function in $y$ and $z$, respectively. By using first-order optimality conditions, we can rewrite this problem as in \eqref{VI-main} with the following variables:
\begin{align*}
    x = \begin{pmatrix} y\\ z \end{pmatrix}, \quad F = \begin{pmatrix} \nabla_y f(y,z)\\ -\nabla_z f(y,z) \end{pmatrix}, \quad g(x) = g_1(y)+g_2(z).
\end{align*}
Furthermore, in many applications of reinforcement learning and game theory, we need to solve a fixed-point problem. For instance, Markov decision processes (MDPs) are a powerful modeling framework in reinforcement learning, where we should solve a fixed point problem, $Tx = x$, for some finite dimensional operator $T$, that is~\eqref{VI-main} with $F = \mbox{Id} - T$ and $g(x) = 0$~\cite{browder1967new}.\\
Several iterative algorithms have been introduced to address VI problems \eqref{VI-main}. For comparison purposes, let us review some recent and closely related existing methods used in system and control applications. For simplicity, let us consider the VI problem formulation in~\eqref{VI-2}.\\
\textbf{Projected Gradient descent (PGD) \cite{nemirovskij1983problem}:} 
\begin{align*}
    x^{k+1} = {\pi}_{\mathcal{A}}(x^{k} - \lambda F(x^k)),
\end{align*}
where $\lambda$ is the stepsize. The convergence of this method is guaranteed for strongly monotone (with a strongly monotone constant $\mu$) and Lipschitz (with a Lipschitz constant $L$) operator with $\lambda \in \left(0, 2\mu/L^2\right)$. {This algorithm is used for equilibrium seeking in aggregative games \cite{belgioioso2021semi} (IEEE TAC, 2021), optimal consensus and resource allocation \cite{huang2023unified} (IEEE TAC, 2023), open-loop Nash equilibrium in linear quadratic dynamic games \cite{benenati2024linear}, and game theoretical approach for generative adversarial network \cite{franci2020game} (IEEE CDC, 2020).}\\
\textbf{Extragradient descent \cite{korpelevich1977extragradient}:} 
\begin{align*}
    & y^{k} = {\pi}_{\mathcal{A}}(x^{k} - \lambda F(x^k)),\\
    & x^{k+1} = {\pi}_{\mathcal{A}}(x^{k} - \lambda F(y^k)),
\end{align*}
where $\lambda$ is the stepsize, and unlike the previous method, the convergence is guaranteed for a Lipschitz 
 and monotone operator (with a Lipschitz constant $L$) with $\lambda \in \left(0, 1/L\right)$. {The authors in \cite{yousefian2014optimal} (IEEE CDC, 2014) implement this method to design robust stochastic extragradient algorithms for solving monotone VIs. Similarly, \cite{huang2023unified} (IEEE TAC, 2023) uses the extra-gradient method to solve the variational inequality problem in optimal consensus and resource allocation problems. The same algorithm is adopted in \cite{xie2022collaborative} (IEEE TPS, 2022) to design an algorithm for the variational inequality problem associated with a collaborative pricing scheme for a power-transportation coupled network.}\\
 %The extragradient method has been extensively studied and improved in various ways. For brevity, we refer interested reader to~\cite{tseng1995linear, korpelevich1977extragradient} for further details.\\
\textbf{Projected Reflected Gradient descent (PrjRef) \cite{malitsky2015projected}:} 
\begin{align*}
    x^{k+1} = {\pi}_{\mathcal{A}}(x^{k} - \lambda F(2x^k - x^{k-1})),
\end{align*}
where $\lambda$ is the stepsize, and the convergence of this method is guaranteed for Lipschitz and monotone operator (with a Lipschitz constant $L$) with $\lambda \in \left(0, (\sqrt{2}-1)/L\right)$. Unlike the extragradient method, PrjRef needs only one projection per iteration. {The authors in \cite{cui2016analysis} (IEEE CDC, 2016) implement this method to design a stochastic monotone VI algorithm under some weak sharpness assumptions. Likewise, this method is adopted in \cite{guo2021variational} (IEEE CDC, 2021) to design an algorithm for solving Bayesian regression game, which is a special class of two-player general-sum Bayesian game. The authors in \cite{franci2021distributed} (ECC, 2021) also use the projected-reflected method to design an algorithm for stochastic generalized Nash equilibrium problems.}\\
\textbf{Golden RAtio ALgorithm (GRAAL) \cite{malitsky2020golden}:} 
\begin{align*}
    & y^{k} = (1-\beta)x^{k} + \beta y^{k-1},\\
    & x^{k+1} = {\pi}_{\mathcal{A}}(y^{k} - \lambda F(x^k)),
\end{align*}
where $\lambda$ is the stepsize and $\beta \in \left(0, {(\sqrt{5}-1)/2}\right]$. The convergence of this method is guaranteed for Lipschitz and monotone operator (with a Lipschitz constant $L$) with $\lambda \in \left(0, 1/{(2\beta L)}\right)$, and it requires one projection per iteration. The stepsize in this method can be chosen adaptively as follows, leading to the Adaptive Golden RAtio ALgorithm (aGRAAL) \cite{malitsky2020golden}:
\begin{align*}
    \lambda_k &= \\
    &\min \left \{(\beta + \beta^2)\lambda_{k-1},  \frac{\|x^k - x^{k-1}\|^2}{4\beta^2\lambda_{k-2}\|F(x^k) - F(x^{k-1})\|^2}, \bar{\lambda}\right\}.
\end{align*}
{This algorithm is applied in \cite{franci2021stochastic} (IEEE TAC, 2021) and \cite{krilavsevic2021extremum} (IEEE CDC, 2021) to design a method for (stochastic) generalized Nash equilibrium in monotone games. The authors in \cite{pantazis2024nash} also use this method in a stochastic portfolio allocation game as a case study for Nash equilibrium seeking in quadratic-bilinear Wasserstein distributionally robust games.}\\
We also refer interested readers to \cite{mignoni2025monviso} for additional algorithms for applications of monotone VI, as well as a ready-to-use Python toolbox.

\textbf{Contribution.} In this paper, we propose two algorithms for solving the monotone variational inequality problem in \eqref{VI-main} that do not require knowledge of a global Lipschitz constant. {Both algorithms are inspired by stability in switched and hybrid systems, where a switched system is asymptotically stable if each subsystem has a strictly decreasing Lyapunov function and switching does not increase it \cite{liberzon2003switching}. Then, based on this context and the application of hybrid system stability in optimization and extremum seeking \cite{goebel2009hybrid,poveda2016hybrid,poveda2018hybrid}, we switch between two algorithms adopted for solving VIs.} Our technical contribution is to show convergence for the first algorithm and the ergodic \(\mathcal{O}(k^{-1})\) convergence rate for the second one. The proposed algorithms reduce dependency on the negative momentum term, previously used in \cite{malitsky2020golden} to ensure boundedness and convergence of iterates, by increasing the momentum parameter in some iterations with the potential to switch the momentum parameter between small (used in aGRAAL) and large values. Using a large momentum parameter in our proposed algorithms (Algorithms \ref{alg1} and \ref{alg2}) brings the iterations closer to the most recent one, allowing us to estimate the local Lipschitz constant of \( F \) more accurately and reducing the frequent use of the negative momentum term which repetitively affects convergence speed \cite{alacaoglu2023beyond}. Briefly speaking, if \( F \) is a Lipschitz and monotone operator, our proposed methods switch between PGD (method without momentum) and aGRAAL (method with the negative momentum) based on certain conditions, along with an adaptive stepsize. Finally, we provide several numerical experiments in which the proposed algorithms consistently outperform the existing state-of-the-art. We note that our method rarely requires additional computations for operator and projection evaluations compared to aGRAAL. However, in the worst case, these computations may need to be performed twice.\\
%The paper is organized as follows. In Section \ref{first adaptive VI}, we introduce the first algorithm for solving \eqref{VI-main} and the theoretical results are explained. Section \ref{second adaptive VI} introduces the second adaptive method for solving \eqref{VI-main}. Finally, several illustrative examples to show the efficiency of our approach is presented in Section \ref{simulation}.
%%%%%%%%%%%%%%%%%%%%%%%%%%%%%%%%%%%%%%%%%%%
\begin{algorithm}[!b]
\caption{Adaptive algorithm for VI (Method 1)}
\label{alg1}
\begin{algorithmic}[1]
\REQUIRE Choose $x^0$, $x^1$, $\bar{\lambda} \gg 0$, $\lambda_0 > 0$, $\phi \in (1,\frac{1+\sqrt{5}}{2}]$, $\theta_0 = 1$, $\rho = \dfrac{1}{\phi} + \dfrac{1}{\phi^2}$, $\text{flg} = 0$, $\bar{k} = 1$
\STATE \textbf{For} {$k = 1,2,\ldots$} \textbf{do}
\STATE Find the stepsize:
\[\lambda_k = \min\left\{\rho \lambda_{k-1}, \dfrac{\phi\theta_{k-1}}{4\lambda_{k-1}} \dfrac{\|x^k - x^{k-1}\|^2}{\|F(x^k) - F(x^{k-1})\|^2}, \bar{\lambda}\right\}\] 
%\[\lambda_k = \min\left\{\rho \lambda_{k-1}, \dfrac{\phi\theta_{k-1}}{4\lambda_{k-1}} \dfrac{\|x^k - x^{k-1}\|^2}{\|F(x^k) - F(x^{k-1})\|^2}, \Bar{\lambda}\right\} \tag{$L_1$} \]
\IF{ ($J(x^k) - J(x^{k-1}) > 0 \,\,\, \land \,\,\, \text{flg} = 1$) \,\, $\lor$ \,\, $\min\{ J_i \}_{i = 0}^{k-1} < J_k + 1/\bar{k}$ }
    \STATE \hspace{25mm} $\bar{x}^{k} = \dfrac{(\phi - 1)x^k + \bar{x}^{k-1}}{\phi}, \, \text{flg} = 0$ 
\ELSE
    \STATE \hspace{25mm} $\bar{x}^{k} = x^k$, $\text{flg} = 1$, $\bar{k} = \bar{k} + 1$
    %\STATE $\text{flg} = 1$
    %\STATE $\bar{k} = \bar{k} + 1$
\ENDIF
\STATE Update the next iteration: \\
\hspace{25mm} $x^{k+1} = \text{prox}_{\lambda_k g}(\bar{x}^k - \lambda_k F(x^k))$
\STATE Update: \qquad\qquad $\theta_{k} = \dfrac{\phi\lambda_k}{ \lambda_{k-1}}$
\STATE Residual computation: \,  $J_{k+1} = x^k - \text{prox}_{g}(x^k - F(x^k))$
\end{algorithmic}
\end{algorithm}
%%%%%%%%%%%%%%%%%%%%%%%%%%%%%%%%%%%%%%%%%%%
%%%%%%%%%%%%%%%%%%%%%%%%%%%%%%%%%%%%%%%%%%%%%
\begin{algorithm}[!t]
\caption{Adaptive algorithm for VI (Method 2)}
\label{alg2}
\begin{algorithmic}[1]
\REQUIRE Choose $x^0$, $x^1$, $\bar{\lambda} \gg 0$, $\lambda_0 > 0$, $\alpha \in (1,\frac{1+\sqrt{5}}{2}]$, $\theta_0 = 1$, $\rho = \dfrac{1}{\alpha} + \dfrac{1}{\alpha^2}$, $\bar \phi \gg \frac{1+\sqrt{5}}{2}$, $\text{sum}_0 ^1 = 0$, $\text{sum}_0 ^2 = 0$, flg = 1, $\phi_0 = \bar{\phi}$.
\STATE \textbf{For} {$k = 1,2,\ldots$} \textbf{do}
\STATE Find the stepsize:
\qquad \qquad \qquad \[\lambda_k = \min\left\{\rho \lambda_{k-1}, \dfrac{\alpha\theta_{k-1}}{4\lambda_{k-1}} \dfrac{\|x^k - x^{k-1}\|^2}{\|F(x^k) - F(x^{k-1})\|^2}, \bar{\lambda}\right\}\]
%\qquad \qquad \qquad \[\lambda_k = \min\left\{\rho \lambda_{k-1}, \dfrac{\alpha\theta_{k-1}}{4\lambda_{k-1}} \dfrac{\|x^k - x^{k-1}\|^2}{\|F(x^k) - F(x^{k-1})\|^2}, \Bar{\lambda}\right\} \tag{$L_2$} \]
\STATE $\bar{x}^{k} = \dfrac{(\phi_k - 1)x^k + \bar{x}^{k-1}}{\phi_k}$
\STATE Update the next iteration:\\
\qquad \qquad \qquad $x^{k+1} = \text{prox}_{\lambda_k g}(\bar{x}^k - \lambda_k F(x^k))$
\STATE Update: \qquad $\theta_{k} = \dfrac{\alpha\lambda_k}{ \lambda_{k-1}}$
\STATE compute the following summations with $\phi_{k+1} = \bar \phi$:\\
\qquad \qquad \qquad $\text{sum}_{k+1}^1 = \text{sum}_k^1$ + \eqref{ls:eq:12} \\
\qquad \qquad \qquad $\text{sum}_{k+1}^2 = \text{sum}_k^2$ + \eqref{ls:eq:13}
\IF{ ($\text{sum}_{k+1}^1 \leq 0 \,\,\, \land \,\,\, \text{flg} = 1$) \,\,$\lor$\,\, ($\text{sum}_{k+1}^2 \leq 0 \,\,\, \land \,\,\, \text{flg} = 0$)}
    \STATE $\phi_{k+1} = \bar\phi$, $\text{flg} = 1$ %\hspace*{\fill} {$\triangleright\,$ \textbf{case \ref{negative summation} or modifying $\phi_{k+1}$ to large value \ref{switching}}}
\ELSE
\IF{$\text{flg} = 1$} 
    \STATE $x^{k+1} = x^k$, $x^{k} = x^{k-1}$, $\bar x^{k} = \bar x^{k-1}$  %\hfill {$\triangleright \,$ \textbf{modifying $\phi_{k+1}$ to small value \ref{switching}}}
    %\STATE $\bar x^k = \bar x^{k-1}$
    \STATE $\phi_{k+1} = \alpha$, $\theta_k = \theta_{k-1}$, $\lambda_{k} = \lambda_{k-1}$ 
    %\STATE $\theta_k = \theta_{k-1}$
    %\STATE $\lambda_{k} = \lambda_{k-1}$
    \STATE $\text{sum}_{k+1}^1 = 0$, $\text{sum}_{k+1}^2 = 0$, $\text{flg} = 0$
    %\STATE $\textbf{flg} = 1$
    \ELSE  
    \STATE $\phi_{k+1} = \alpha$ %\hfill {$\triangleright\,$ \textbf{case \ref{positive summation}: continue and update}} 
    \STATE $\text{sum}_{k+1}^2 = \text{sum}_k^2$ + (\eqref{ls:eq:13} with $\phi_{k+1}  = \alpha$) %\hspace*{8mm} {\textbf{ $\textnormal{sum}_{k+1}^2$ with small $\phi_{k+1}$, if $\textnormal{sum}_{k+1}^2$ is}}
    \STATE $\text{sum}_{k+1}^1 = 0$ %\hspace*{49mm} {\textbf{not negative with large $\phi_k$ and \textnormal{flg} $=0$}}
\ENDIF 
\ENDIF
\end{algorithmic}
\end{algorithm}
%%%%%%%%%%%%%%%%%%%%%%%%%%%%%%%%%%
\setlength{\textfloatsep}{5pt}
\textbf{Notation.} Let $\mathcal{V}$ be the finite-dimensional real vector space with the standard inner product $\langle \cdot, \cdot \rangle$ and $\ell_p$-norm  $\|\cdot\|_p$ (by $\|\cdot\|$, we mean the Euclidean standard 2-norm). We also denote the $\pi_{\mathcal{A}}$ for the metric projection onto set $\mathcal{A}$ ($\pi_{\mathcal{A}}(x) = \arg\min_{y \in \mathcal{A}} \|x - y\|$), $\delta_{\mathcal{A}}$ the indicator function of set $\mathcal{A}$, $\text{dist}(x,\mathcal{A})$ the distance from $x$ to set $\mathcal{A}$ ($\text{dist}(x,\mathcal{A}) = \| \pi_{\mathcal{A}}(x) - x\|$), and $\mathbb{B}(\Tilde{x},r)$ a closed ball with center $\Tilde{x}$ and radius $r>0$. The operator $F$ is $L$-Lipschitz, if there is $L>0$ such that for all $x,y \in \mathcal{V}$ we have $\|F(x) - F(y)\| \leq L\|x-y\|$.  
%following inequality hold.
\begin{comment}
\begin{equation} \label{eq3}
\|F(x) - F(y)\| \leq L\|x-y\|
\end{equation}
\end{comment}
Furthermore, $F$ is locally Lipschitz, if it is Lipschitz over any compact set of its domain. The operator $F$ is monotone if $\langle F(x) - F(y), x - y \rangle \geq 0$ for all $x,y \in \mathcal{V}$
\begin{comment}
\begin{equation} \label{eq4}
\langle F(x) - F(y), x - y \rangle \geq 0
\end{equation}
\end{comment}
and it is called strongly monotone with constant $\mu > 0$ if 
$\langle F(x) - F(y), x - y \rangle \geq \mu \|x-y\|^2$ for all $x,y \in \mathcal{V}$.
\begin{comment}
the following inequality holds $x,y \in \mathcal{V}$
\begin{equation} \label{eq5}
\langle F(x) - F(y), x - y \rangle \geq \mu \|x-y\|^2
\end{equation}
\end{comment}
%We say that $F$ satisfies the Minty variational inequality problem if there exists $\hat{x} \in \mathcal{A}$ such that the following inequality holds for all $x \in \mathcal{V}$
%\begin{align}\label{Minty}
%    \langle F(x), x - \hat{x} \rangle + g(x) - g(\hat{x}) \geq 0 \tag{Minty VI}
%\end{align}
%Generally (if $F$ is continuous), the solution set of the Minty VI ($\mathcal{S_{\text{Minty}}^{\text{VI}}}$) is a subset of the solution set of the main VI problem \eqref{VI-main}. 
The prox operator of a function $g \colon \mathcal{V} \to \mathbb{R}$ is defined as $\text{prox}_{g}(x) = \arg \min_{u} \{g(u) + \|u-x\|^2/2\}$. A function is ``prox-friendly'' if the prox operator is available (computationally or explicitly). The following equations are useful and commonly used in the proofs \cite{facchinei2003finite}:
\begin{subequations}
\begin{align} 
    y = \text{prox}_g x &\Longleftrightarrow \langle y - x, z - y \rangle \geq \nonumber\\
    & \qquad\qquad\qquad\quad g(y) - g(z), \quad \forall z \in \mathcal{V} \label{eq7}\\
    \|a x + (1-&a)y\|^2 = a\|x\|^2 + (1-a)\|y\|^2 \nonumber\\
    - a &(1-a)\|x-y\|^2. \quad \forall x,y \in \mathcal{V}, \, \forall a \in \mathbb{R} \label{eq8}
\end{align}
\end{subequations}
\begin{comment}
\begin{lemma}[Sequence convergence {\cite[Lemma~1]{malitsky2020golden}}]\label{lemma1}
      If ${x^k} \in \mathcal{V}$ is a bounded sequence, and $\lim\limits_{k\rightarrow \infty} (x^k - x)$ exists, where $x$ is a cluster point of the sequence ${x^k}$, then $x^k$ is convergent.
\end{lemma}
\end{comment}
%%%%%%%%%%%%%%%%%%%%%%%%%%%%%%%%%%%%%%%%%%%%%%%%%%%%%%%%%%%%%
\section{Preliminaries and First Algorithm}\label{first adaptive VI}
In this section, we first present the main theorem, which helps establish the boundedness and convergence of the iterations with a variable momentum parameter for the algorithms whose general forms are given in Algorithms \ref{alg1} and \ref{alg2}. We then describe our first algorithm, which follows the PGD and aGRAAL frameworks, differing only in the choice of the momentum parameter determined by conditions ensuring a sufficient decrease in the error bound. \\
Before proceeding with the theorem, let us define the merit function $\Psi(x,y) := \langle F(x), y - x \rangle + g(y) - g(x)$, which is convex with respect to $y$. It can be easily seen that \eqref{VI-main} is equivalent to finding $x^* \in \mathcal{V}$ such that $\Psi(x^*,x) \geq 0, \, \forall x\in \mathcal{V}$.
\begin{theorem}[Variable momentum in aGRAAL]\label{main-aGRAAL}
Let $F\colon \dom g \to \mathcal{V}$ be locally Lipschitz and monotone operator. Then $\left(x^k\right)_{k \in \mathbb{N}}$ and $\left(\bar x^k\right)_{k \in \mathbb{N}}$, generated by Algorithms~\ref{alg1}-\ref{alg2}, satisfy the following inequality:
\begin{align}\label{ls:eq:11}
    &\frac{\phi_{k+1}}{\phi_{k+1} -1 } \n{\x^{k+1}-x}^2 +
    \frac{\theta_k}{2}\n{x^{k+1}-x^k}^2 + 2\la_k \Psi(x, x^k) \nonumber\\ 
    &\leq  \frac{\phi_{k+1}}{\phi_{k+1} -1 }
    \n{\x^{k}-x}^2 + \frac{\theta_{k-1}}{2}\n{x^{k}-x^{k-1}}^2 \nonumber\\    
   &-\dfrac{\lambda_k}{\lambda_{k-1}}\phi_k \n{x^k-\x^k}^2 +\bigl(\dfrac{\lambda_k}{\lambda_{k-1}}\phi_k - 1 - \frac{1}{\phi_{k+1}}\bigr) \n{x^{k+1}-\x^{k}}^2 \nonumber\\
   & - \bigl(\dfrac{\lambda_k}{\lambda_{k-1}}\phi_k - \theta_k\bigr)\n{x^{k+1}-x^k}^2.
\end{align}
\end{theorem}
\begin{proof}
Let $x \in \mathcal{V}$ be arbitrary. Now consider Algorithms \ref{alg1} and \ref{alg2}, where ${x}^k$ and $\Bar{x}^k$ are updated as follows:
\begin{align*}
    \Bar{x}^k = \dfrac{(\phi_k - 1)x^k + \Bar{x}^{k-1}}{\phi_k}, \quad x^{k+1} = \text{prox}_{\lambda_k g}(\Bar{x}^k - \lambda_k F(x^k)).
\end{align*}
Then by using \eqref{eq7}, we have
\begin{align}
   \lr{x^{k+1}-\x^k+\la_k F(x^k), x -& x^{k+1}}  \geq \nonumber\\
   \la_k&\left(g(x^{k+1})-g(x)\right),  \label{ls:eq:4}\\
  \lr{x^{k}-\x^{k-1}+\la_{k-1} F(x^{k-1}), &x^{k+1} - x^{k}}  \geq \nonumber\\
  \la_{k-1}&\left( g(x^{k})-g(x^{k+1})\right).    \label{ls:eq:5}
\end{align}

Multiplying \eqref{ls:eq:5} by $\frac{\la_k}{\la_{k-1}} \geq 0$ and using that $x^k-\x^{k-1}={\phi_k}(x^k-\x^{k})$, we obtain
\begin{align}
    \label{ls:eq:6}
    \lr{\dfrac{\lambda_k}{\lambda_{k-1}}\phi_k(x^{k}-\x^{k})&+\la_k F(x^{k-1}), x^{k+1} - x^{k}} \nonumber\\
    & \geq \la_k \left( g(x^{k})-g(x^{k+1})\right).
\end{align}
The summation of~\eqref{ls:eq:4} and~\eqref{ls:eq:6} gives us
\begin{align}
    \label{ls:eq:7}
    &\lr{x^{k+1}-\x^k, x - x^{k+1}} + \dfrac{\lambda_k \phi_k }{\lambda_{k-1}}\lr{x^{k}-\x^{k}, x^{k+1} -x^{k}} \nonumber\\
    & + \la_k \lr{F(x^k)-F(x^{k-1}), x^k-x^{k+1}} \geq \nonumber\\
    &\la_k \lr{F(x^k), x^k-x} + \la_k \left(g(x^{k})-g(x)\right) \geq \nonumber \\ 
     &\la_k \Big[\lr{F(x), x^k-x} + g(x^{k})-g(x)\Big] = \la_k \Psi(x, x^k)    .
\end{align}
Expressing the first two terms in~\eqref{ls:eq:7} through norms leads to
\begin{align}
    \label{ls:eq:8}
  \n{x^{k+1}-x}^2 \leq \n{&\x^k-x}^2 - \n{x^{k+1}-\x^k}^2
     \nonumber \\
  &  + 2\la_k\lr{F(x^k)-F(x^{k-1}),x^k-x^{k+1}}\nonumber\\
  & \hspace{-25mm} +  \dfrac{\lambda_k}{\lambda_{k-1}}\phi_k \left(\n{x^{k+1} -\x^k}^2 -
  \n{x^{k+1}-x^k}^2-\n{x^k-\x^k}^2\right) \nonumber \\
  & - 2\la_k \Psi(x, x^k).
\end{align}  
Similarly to~\eqref{eq8}, we have
\begin{align}\label{ls:eq:identity}
    \n{x^{k+1}&-x}^2 = \frac{\phi_{k+1}}{\phi_{k+1}-1}\n{\x^{k+1}-x}^2 \nonumber\\
    &-\frac{1}{\phi_{k+1}-1}\n{\x^{k}-x}^2+\frac{1}{\phi_{k+1}}\n{x^{k+1}-\x^k}^2.
\end{align}
By combining this with~\eqref{ls:eq:8}, we obtain
\begin{align}
    \label{ls:eq:9}
    &\frac{\phi_{k+1}}{\phi_{k+1} -1 } \n{\x^{k+1}-x}^2  \leq  \frac{\phi_{k+1}}{\phi_{k+1} -1 }
    \n{\x^{k}-x}^2 + \nonumber\\
    &\bigl(\dfrac{\lambda_k}{\lambda_{k-1}}\phi_k - 1 -
    \frac{1}{\phi_{k+1}}\bigr) \n{x^{k+1}-\x^{k}}^2   - 2\la_k \Psi(x, x^k) \nonumber\\
    &-  \dfrac{\lambda_k}{\lambda_{k-1}}\phi_k \bigl(\n{x^{k+1}-x^k}^2 + \n{x^k-\x^k}^2\bigr) \nonumber \\
    &+ 2\la_k\lr{F(x^k)-F(x^{k-1}),x^k-x^{k+1}}.
\end{align}
Using the stepsize update rule and Holder's inequality, the last term on the right-hand side of~\eqref{ls:eq:9} upper bounded by
\begin{align}
    \label{ls:eq:10}
     2\la_k\lr{F(x^k)-F(x^{k-1}),x^k-x^{k+1}} &\leq \nonumber \\
     2\la_k\n{F(x^k)-F(x^{k-1})}\n{x^k-x^{k+1}} &\leq \nonumber \\
     \sqrt{\theta_k \theta_{k-1}}\n{x^k-x^{k-1}}\n{x^k-x^{k+1}} &\leq \nonumber \\
     \frac{\theta_k}{2} \n{x^{k+1}-x^k}^2 +  \frac{\theta_{k-1}}{2} \n{x^{k}-x^{k-1}&}^2.
\end{align}
Finally, by applying the obtained estimate to~\eqref{ls:eq:9}, we conclude the result of the theorem.
\begin{comment}
\begin{align}
    \label{ls:eq:11}
    &\frac{\phi_{k+1}}{\phi_{k+1} -1 } \n{\x^{k+1}-x}^2  +
    \frac{\theta_k}{2}\n{x^{k+1}-x^k}^2 +  2\la_k \Psi(x, x^k) \leq \nonumber\\ 
    &\frac{\phi_{k+1}}{\phi_{k+1} -1 } \n{\x^{k}-x}^2 +  \frac{\theta_{k-1}}{2}\n{x^{k}-x^{k-1}}^2 -\dfrac{\lambda_k}{\lambda_{k-1}}\phi_k \n{x^k-\x^k}^2 \nonumber\\
    &+\bigl(\dfrac{\lambda_k}{\lambda_{k-1}}\phi_k - 1 - \frac{1}{\phi_{k+1}}\bigr) \n{x^{k+1}-\x^{k}}^2 \nonumber\\
    &- \bigl(\dfrac{\lambda_k}{\lambda_{k-1}}\phi_k - \theta_k\bigr)\n{x^{k+1}-x^k}^2.
\end{align}
\end{comment}
\end{proof}
By controlling the right-hand-side of \eqref{ls:eq:11}, we can prove the boundedness and convergence of the sequence $\left(x^k\right)_{k \in \mathbb{N}}$.
Next, we aim to maintain the negativity of the last three terms of the right-hand-side of \eqref{ls:eq:11} while ensuring that $\phi_k$ attains a sufficiently large value which makes $\bar{x}^k$ closer to the current iterate ${x}^k$ instead of $\bar{x}^{k-1}$. Subsequently, we elaborate on two methods devised for achieving this objective.

%\begin{figure}[h]
%    \centering
%    \includegraphics[width=0.3\textwidth]{ECC2025/images/method1.pdf}
%    \caption{The procedure of the first method.}
%    \label{fig1}
%\end{figure}
%%%%%%%%%%%%%%%%%%%%%%%%%%%%%%%%%%%%%%%%%%%
%\begin{wrapfigure}{r}{0.53\linewidth}
%\captionsetup{justification=centering}
%     \centering
%           \includegraphics[width=0.33\textwidth]{ECC2025/images/VI_fig12.pdf}
%           \caption{The procedure of the first method.}
%           \label{fig1}
%\end{wrapfigure}
%%%%%%%%%%%%%%%%%%%%%%%%%%%%%%%%%%%%%%%%%%%
\textbf{Algorithm \ref{alg1}.} The idea of Algorithm \ref{alg1} is to alternate between the algorithm with a small $\phi \in \left(1, \frac{1+\sqrt{5}}{2}\right]$ and the one without momentum (or equivalently $\phi = \infty$) based on the residual evaluation, used as a measure of performance in VI \cite[Proposition~1.5.8]{facchinei2003finite}. The use of small $\phi$ ultimately leads to the convergence of the residual to zero due to the negativity of the three rightmost terms in \eqref{ls:eq:11} \cite[Theorem~2]{malitsky2020golden}. We initiate the algorithm without the momentum term, and by computing the residual, $J_k = \| x^k - \text{prox}_{g}({x}^k - F(x^k)) \| $, in each iteration, we continue without $\phi$ if the residual is decreasing. Conversely, if the residual is not decreasing, we switch $\phi$ to the small value until the residual becomes smaller than the minimum residual achieved so far plus $\frac{1}{\bar{k}}$, where $\bar{k}$ denotes the number of times switching has occurred so far.
%iteration index indicating the number of iterations when switching occurs.
%%%%%%%%%%%%%%%%%%%%%%%%%%%%%%%%%%%%%%%%%%%
\begin{figure}[!h]
\captionsetup{justification=centering}
     \centering
           \includegraphics[width=0.38\textwidth]{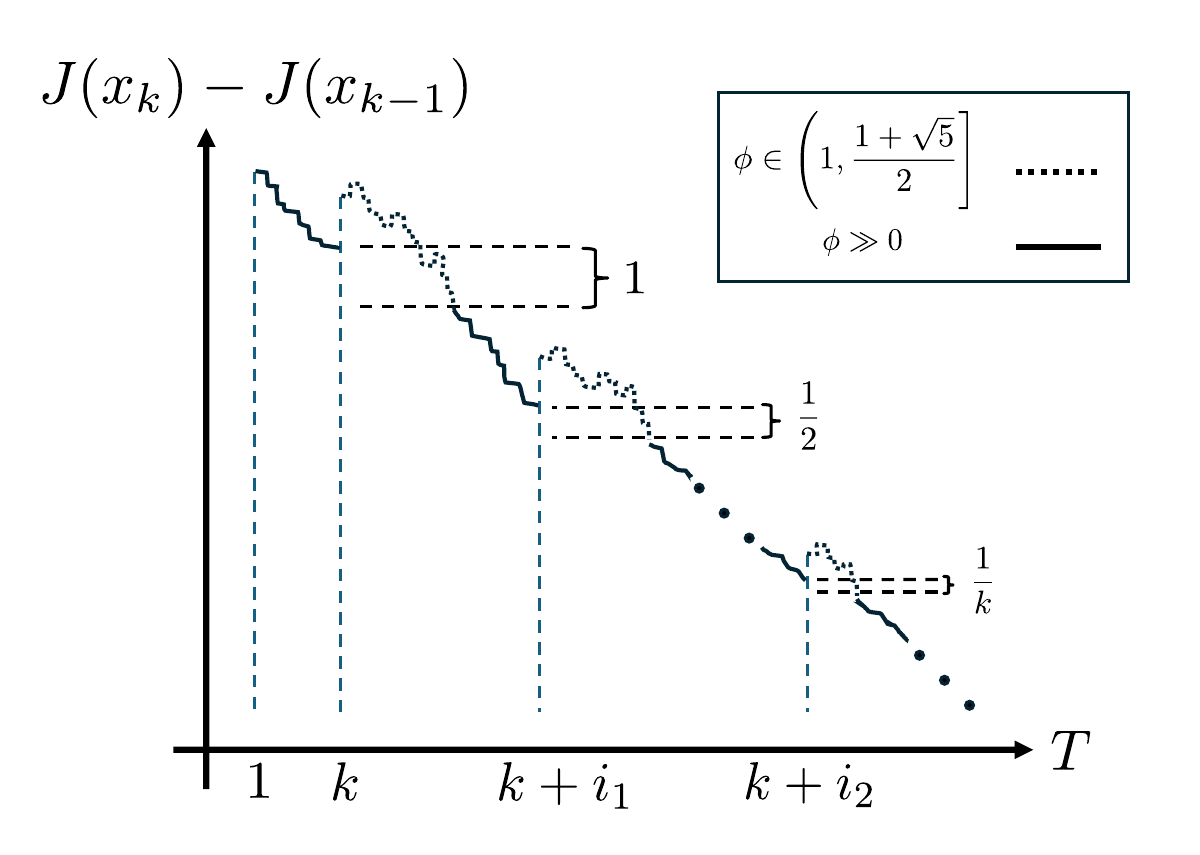}
           \caption{\footnotesize Residual variation induced by Algorithm \ref{alg1}.}
           \label{fig1}
\end{figure}
%%%%%%%%%%%%%%%%%%%%%%%%%%%%%%%%%%%%%%%%%%%
%\begin{figure}[ht]
%   \centering
%   \captionsetup{justification=centering}
%   \subfloat{\includegraphics[scale=0.70]{images/images_old/VIfig1.pdf}}
%   \caption{The procedure of second method.}
%       \hspace{1em}
%    \label{fig1}
%\end{figure}
%%%%%%%%%%%%%%%%%%%%%%%%%%%%%%%%%%%%%%%%%%%
It is noteworthy that the use of a small $\phi$ may result in non-monotone changes in the residual, and we refrain from altering $\phi$ until the residual becomes smaller than the minimum residual achieved so far plus the non-summable term. Figure \ref{fig1} illustrates how Algorithm \ref{alg1} operates: The alteration of $\phi$ is observed when the residual decreases sufficiently, ensuring convergence due to the fact that $\sum_k 1/\bar{k} \rightarrow \infty$.% as the number of iterations increase.
\section{An efficient switching algorithm}\label{second adaptive VI}
In this section, we analyze the convergence of Algorithm \ref{alg2} for solving \eqref{VI-main}, which follows the aGRAAL method. Differently from aGRAAL, the momentum parameter is not fixed, and in fact, in our numerical experience, it has a large value in many iterations, which supports the acceleration of the algorithms. Now, by employing \eqref{ls:eq:11}, we use a simple analysis to control the right-hand side of \eqref{ls:eq:11} and aim to maintain the negativity of the right-hand side while ensuring that $\phi_k$ attains a sufficiently large value.

\textbf{Algorithm \ref{alg2}.}
The algorithm is initiated with a large value for $\phi_k$, and the summation of~\eqref{ls:eq:12} is computed after each iteration (with large value of $\phi_{k+1}$). If the resulting summation is negative, the algorithm proceeds with the initial large value of $\phi_k$. Conversely, if the summation is not negative, the algorithm is reset (by restarting, $x^{k+1}$ that is generated by large $\phi_k$ and other parameters with indices $k$ are not considered as a new iteration and variables, lines 11-13 of Algorithm~\ref{alg2}), and $\phi_k$ is chosen from the interval $\left(1, \frac{1+\sqrt{5}}{2}\right]$.
\begin{align}\label{ls:eq:12}
%\color{red} 
\frac{\theta_{k-1}}{2}&\n{x^{k}-x^{k-1}}^2 -\dfrac{\lambda_k}{\lambda_{k-1}}\phi_k \n{x^k-\x^k}^2 \nonumber \\
+ \bigl(\dfrac{\lambda_k}{\lambda_{k-1}}&\phi_k - 1 - \frac{1}{\phi_{k+1}}\bigr) \n{x^{k+1}-\x^{k}}^2  \nonumber\\
 %\color{red}  
- \bigl(\dfrac{\lambda_k}{\lambda_{k-1}}&\phi_k - \theta_k\bigr)\n{x^{k+1}-x^k}^2  - \frac{\theta_k}{2}\n{x^{k+1}-x^k}^2.
\end{align}
After restarting, the following equation is examined in each iteration (with a large value of $\phi_{k+1}$)
\begin{align}\label{ls:eq:13}
\color{black} -\dfrac{\lambda_k \phi_k}{\lambda_{k-1}} \n{x^k-\x^k}^2 &+\bigl(\dfrac{\lambda_k\phi_k}{\lambda_{k-1}} - 1 - \frac{1}{\phi_{k+1}}\bigr) \n{x^{k+1}-\x^{k}}^2 \nonumber\\
&- \bigl(\dfrac{\lambda_k\phi_k}{\lambda_{k-1}} - \theta_k\bigr)\n{x^{k+1}-x^k}^2.
\end{align}
If the computed summation is negative, then the algorithm employs the large $\phi$ once more for the next iterations; conversely, if the summation is not negative, the algorithm persists with a small value of $\phi$. In this context, three scenarios are contemplated for Algorithm \ref{alg2}
\begin{enumerate}[label=(\roman*), itemsep = 0mm, topsep = -0mm, leftmargin = 7mm]
    \item\label{negative summation} $\textbf{Always negative summation:}$ By telescoping \eqref{ls:eq:11} the summation of \eqref{ls:eq:12} is always negative; therefore, $x^k$ are bounded and $x^k \rightarrow x^*$ if $k \rightarrow \infty$.
    \item\label{positive summation} $\textbf{Always positive summation:}$ %In this scenario the summation in \eqref{ls:eq:13} is positive and according to the Algorithm \ref{alg2} 
    We always have $\phi \in (1, \frac{1+\sqrt{5}}{2}]$, thus we obtain the same algorithm as in \cite[Algorithm~1]{malitsky2020golden}.
    \item\label{switching} $\textbf{Switching between small and large}$ $\phi$ $\textbf{:}$ If $\phi_k$ is small and by modifying $\phi_{k+1}$ to a larger value, \eqref{ls:eq:13} becomes negative, we adjust $\phi_{k+1}$ to a larger value in the subsequent step. Then, the inequality $\frac{\phi_{k+1}}{\phi_{k+1} - 1} \leq \frac{\phi_{k}}{\phi_{k} - 1}$ holds, and \eqref{ls:eq:11} in two steps is as follows:
\begin{align}
    \label{ls:eq:14}
    &(\frac{\phi_{k}}{\phi_{k} -1 } - \frac{\phi_{k+1}}{\phi_{k+1} -1 }) \n{\x^{k}-x}^2 + \frac{\phi_{k+1}}{\phi_{k+1} -1 } \n{\x^{k}-x}^2 \nonumber\\
    & + \frac{\theta_{k-1}}{2}\n{x^{k}-x^{k-1}}^2 \color{black} + 2\la_{k-1} \Psi(x, x^{k-1})  \nonumber \\
    &\leq \frac{\phi_{k}}{\phi_{k} -1 } \n{\x^{k-1}-x}^2 +  \frac{\theta_{k-2}}{2}\n{x^{k-1}-x^{k-2}}^2\nonumber\\
    &  \color{black} +\bigl(\dfrac{\lambda_{k-1}}{\lambda_{k-2}}\phi_{k-1} - 1 - \frac{1}{\phi_{k}}\bigr) \n{x^{k}-\x^{k-1}}^2 \nonumber\\
    & -\dfrac{\lambda_{k-1}}{\lambda_{k-2}}\phi_{k-1} \n{x^{k-1}-\x^{k-1}}^2 \nonumber\\
   & - \bigl(\dfrac{\lambda_{k-1}}{\lambda_{k-2}}\phi_{k-1} - \theta_{k-1}\bigr)\n{x^{k}-x^{k-1}}^2.
\end{align}
\begin{align}
    \label{ls:eq:15}
    &\frac{\phi_{k+1}}{\phi_{k+1} -1 } \n{\x^{k+1}-x}^2  +
    \frac{\theta_k}{2}\n{x^{k+1}-x^k}^2 +  2\la_k \Psi(x, x^k) \nonumber \\
    & \leq \frac{\phi_{k+1}}{\phi_{k+1} -1 } \n{\x^{k}-x}^2 + \frac{\theta_{k-1}}{2}\n{x^{k}-x^{k-1}}^2 \nonumber \\
    &-\dfrac{\lambda_k}{\lambda_{k-1}}\phi_k \n{x^k-\x^k}^2 - \bigl(\dfrac{\lambda_k}{\lambda_{k-1}}\phi_k - \theta_k\bigr)\n{x^{k+1}-x^k}^2   \nonumber\\
     & +\bigl(\dfrac{\lambda_k}{\lambda_{k-1}}\phi_k - 1 - \frac{1}{\phi_{k+1}}\bigr) \n{x^{k+1}-\x^{k}}^2,
\end{align}
where in the first line of \eqref{ls:eq:14} we add and subtract $\frac{\phi_{k+1}}{\phi_{k+1} -1 } \n{\x^{k}-x}^2$. However, if $\phi_k$ is large and the summations of \eqref{ls:eq:12} is not negative, the algorithm should be reset with a smaller $\phi_k$. \\
Let us assume we switch to the large $\phi$ in the ${k}^{\text{th}}$ iteration and after $i+1$ steps, we change $\phi$ to a small value. In this case, the condition "$\text{sum}_{k+1}^1 \leq 0$" in Algorithm \ref{alg2} (negative summation of \eqref{ls:eq:12}) ensures that $\n{\x^{k}-x}^2 \geq \n{\x^{k+i}-x}^2$ while $\frac{\phi_{k}}{\phi_{k} -1 } - \frac{\phi_{k+1}}{\phi_{k+1} -1 } = -(\frac{\phi_{k+i}}{\phi_{k+i} -1 } - \frac{\phi_{k+i+1}}{\phi_{k+i+1} -1 })$. Therefore, \eqref{ls:eq:11} in two steps can be expressed as follows:
\begin{align}    
    &(\frac{\phi_{k+i}}{\phi_{k+i} -1 } - \frac{\phi_{k+i+1}}{\phi_{k+i+1} -1 }) \n{\x^{k+i}-x}^2 \nonumber \\
    & + \frac{\phi_{k+i+1}}{\phi_{k+i+1} - 1} \n{\x^{k+i}-x}^2 + \frac{\theta_{k+i-1}}{2}\n{x^{k}-x^{k+i-1}}^2 \nonumber \\ 
    & + 2\la_{k+i-1} \Psi(x, x^{k+i-1}) \leq  \frac{\phi_{k+i}}{\phi_{k+i} - 1} \n{\x^{k+i-1}-x}^2 \nonumber\\
    & + \frac{\theta_{k+i-2}}{2}\n{x^{k+i-1}-x^{k+i-2}}^2 \nonumber \\
    & -\dfrac{\lambda_{k+i-1}}{\lambda_{k+i-2}}\phi_{k+i-1} \n{x^{k+i-1}-\x^{k+i-1}}^2 \nonumber\\
    & + \bigl(\dfrac{\lambda_{k+i-1}}{\lambda_{k+i-2}}\phi_{k+i-1} - 1 - \frac{1}{\phi_{k+i}}\bigr) \n{x^{k+i}-\x^{k+i-1}}^2 \nonumber \\
    & - \bigl(\dfrac{\lambda_{k+i-1}}{\lambda_{k+i-2}}\phi_{k+i-1} - \theta_{k+i-1}\bigr)\n{x^{k+i}-x^{k+i-1}}^2. \label{ls:eq:16} \\
    &\frac{\phi_{k+i+1}}{\phi_{k+i+1} -1 } \n{\x^{k+i+1}-x}^2  +
    \frac{\theta_{k+i}}{2}\n{x^{k+i+1}-x^{k+i}}^2 \nonumber \\
    &+ 2\la_{k+i} \Psi(x, x^{k+i}) \leq \frac{\phi_{k+i+1}}{\phi_{k+i+1} -1 } \n{\x^{k+i}-x}^2 \nonumber \\
    &+ \frac{\theta_{k+i-1}}{2}\n{x^{k+i}-x^{k+i-1}}^2 \nonumber \\
    & -\dfrac{\lambda_{k+i}}{\lambda_{k+i-1}}\phi_{k+i} \n{x^{k+i}-\x^{k+i}}^2 \nonumber \\
    &+\bigl(\dfrac{\lambda_{k+i}}{\lambda_{k+i-1}}\phi_{k+i} - 1 - \frac{1}{\phi_{k+i+1}}\bigr) \n{x^{k+i+1}-\x^{k+i}}^2 \nonumber\\
    &- \bigl(\dfrac{\lambda_{k+i}}{\lambda_{k+i-1}}\phi_{k+i} - \theta_{k+i}\bigr)\n{x^{k+i+1}-x^{k+i}}^2, \label{ls:eq:17}
\end{align}
where in the first line of \eqref{ls:eq:16} we add and subtract $\frac{\phi_{k+i+1}}{\phi_{k+i+1} -1 } \n{\x^{k+i}-x}^2$. Then by telescoping \eqref{ls:eq:11} (in both cases, whether switching from a small $\phi$ to a large one \eqref{ls:eq:14} and \eqref{ls:eq:15}, or switching from a large $\phi$ to a small one \eqref{ls:eq:16} and \eqref{ls:eq:17}), we drive \eqref{conv ineq}. More precisely, the conditions in line 7 of Algorithm \ref{alg2} ensure that, by telescoping \eqref{ls:eq:11}, we obtain similar coefficient terms on the right and left-hand side of successive lines of \eqref{ls:eq:11} (e.g., the leftmost terms in \eqref{ls:eq:14} and \eqref{ls:eq:16} can be removed by telescoping the inequalities, and we have similar terms on the right and left-hand sides of two successive inequalities)
%, like $\frac{\phi_{k+1}}{\phi_{k+1} -1 } \n{\x^{k}-x}^2  + \frac{\theta_{k-1}}{2}\n{x^{k}-x^{k-1}}^2$ on the left-hand side of \eqref{ls:eq:14} and the right-hand side of \eqref{ls:eq:15}),
which allows us to \textit{point-wise} remove the similar terms and obtain inequality \eqref{conv ineq} after $T$ iterations, as follows:
%the same inequality as in \cite{malitsky2020golden} (equation (35) in Theorem 2), which provides us with convergence to the solution of \eqref{VI-main}. 
%In particular, by telescoping \eqref{ls:eq:11} for $T$ iterations, we obtain:
\begin{align}\label{conv ineq}
     &\frac{\phi_{T}}{\phi_{T} -1 } \n{\x^{T}-x}^2 +
    \frac{\theta_{T-1}}{2}\n{x^{T}-x^{T-1}}^2 + 2 \sum_{i=1}^T \la_i \Psi(x, x^i)  \nonumber\\ 
    & \quad \leq \frac{\phi_2}{\phi_2 -1 }\n{\x^{1}-x}^2 + \frac{\theta_{0}}{2}\n{x^{1}-x^{0}}^2 + D,
\end{align}
where $D$ is a non-positive constant which is summation of the three negative rightmost terms in~\eqref{ls:eq:11}. Note that $T$ in~\eqref{conv ineq} is not exactly the number of projections or operator evaluations in Algorithm~\ref{alg2}. In more detail, if we are in case~\ref{negative summation} and always continue with large $\phi$, then the number of projections and operator evaluations is exactly \( T \). In the worst-case scenario, we have case~\ref{positive summation}, where the current sequence should regenerate with small $\phi$. In this situation, the number of projections and operator evaluations is $2T$. Finally, if the sequence is generated by switching between large and small $\phi$ (case~\ref{switching}), then the number of projections and operator evaluations is between $T$ and $2T$. It is also worth noting that, in practice, the number of projections and operator evaluations is close to $T$ (see Section~\ref{simulation}).
\end{enumerate}
Similarly to \cite{malitsky2020golden}, we can prove the ergodic convergence rate based on \eqref{conv ineq}. The following theorem indicate the convergence properties of Algorithm \ref{alg2}. 
%To keep it short, we skip the proof and refer interested readers to~\cite{malitsky2020golden} for further details.

\begin{theorem}[Ergodic convergence]\label{Ergodic convergence}
Let $X_k$ be the ergodic sequence $X_k = {\sum_{i = 1}^k \lambda_i x^{i}}/{\sum_{i = 1}^k \lambda_i}$ and $e_r(y) = \max_{x \in \mathcal{U}} \, \Psi(x,y) \quad \forall y \in \mathcal{V}$, where $\mathcal{U} = \dom \, g\cap \mathbb{B}(\hat{x},r)$ and $\hat{x} \in \dom g$. Then, we obtain the $\mathcal{O}(k^{-1})$ convergence rate for the ergodic sequence $X_k$. More precisely we have
\begin{align*}
    e_r(X_k) = \max\limits_{x \in \mathcal{U}} \, \Psi(x, X_k) &\leq \dfrac{M}{k},
    % \dfrac{\max\limits_{x \in \mathcal{U}} \, \left(\sum_{i = 1}^k \Psi(x, x^{i}) \right)}{\sum_{i = 1}^k \lambda_i} \leq \dfrac{M}{\sum_{i = 1}^k \lambda_i}.
\end{align*}
where $M > 0$ is some constant dominates the right-hand side of \eqref{conv ineq} for all $x \in \mathcal{U}$, in particular $\sum_{i=1} \la_i \Psi(x, x^i) \leq M$.
\end{theorem}
\begin{proof}
    See Appendix.
\end{proof}
Algorithm \ref{alg2} consists of three parts. Line 8 handles either the case of \ref{negative summation} or modifies $\phi_{k+1}$ to a large value (case \ref{switching}). In lines 11-13, the algorithm adjusts $\phi_{k+1}$ to a small value (cases \ref{positive summation} or \ref{switching}). Note that in lines 11–13, the generated \( x^{k+1} \) is removed, and we go one step back to reset the setup. We then use \( x^k \), \( x^{k-1} \), and \( \bar{x}^{k-1} \) with an updated setup to generate a new \( x^{k+1} \). Finally, lines 15-17 correspond to case \ref{positive summation}, where we continue by updating $\text{sum}_{k+1}^2$ with a small $\phi_{k+1}$ if $\text{sum}_{k+1}^2$ is non-negative with a large $\phi_k$ and $\text{flg}=0$.
%%%%%%%%%%%%%%%%%%%%%%%%%%%%%%%%%%%%%%%%%%%%%%%%%%%%%%%%%%%%%%
\section{Numerical simulations} \label{simulation}
We demonstrate the performance of Algorithm~\ref{alg1} and \ref{alg2} on six classes of VI problems studied in the literature: \ref{nash} Nash–Cournot equilibrium, \ref{sparse-l-r} sparse logistic regression, \ref{2-zero-sum} Two-player Zero Sum Game, \ref{mdp pr} Markov decision processes, \ref{s-monotone-vi} strongly affine monotone operator, and \ref{nonmonotone vi} VI problem with non-monotone operator. To evaluate the performance of our proposed algorithms, we compare their residual, used as a measure of performance in VI, with the residuals of the following methods from the literature throughout the iterations: (i) Projected Gradient descent (PrGD), (ii) projected reflected Gradient descent (PrRefGD), and (iii) adaptive Golden ratio (aGRAAL), a relatively recent method for monotone variational inequality and the closest in spirit to our proposed methods. To be more fair, we plot the residual ($y$-axis) against the fixed number of \textit{operator evaluation} ($x$-axis) in all our figures. We set $\phi = 1.5$, and $\lambda_0 = \bar{\lambda} = 1$ in Algorithm \ref{alg1} and aGRAAL, and in Algorithm \ref{alg2}, $\bar{\phi} = 1$, $\alpha = 1.5$, and $\lambda_0 = \bar{\lambda} = 1$. The stepsizes for PrGD and PrRefGD are selected in each problem based on the Lipschitz constant of the underlying operator or the largest values that prevent divergence. Note that projection operators in all examples are evaluated using the solver \texttt{OSQP solver}\footnote{https://github.com/osqp/osqp}.
\begin{enumerate}[label=(\arabic*), itemsep = 0mm, topsep = 0mm, leftmargin = 5mm]
\item \textbf{Nash–Cournot equilibrium problem \cite{facchinei2003finite}. \label{nash}}
A variational inequality that corresponds to the Nash–Cournot equilibrium is find $x^* = (x_1^*,\dots, x_n^*)\in \R^{n}_+$
\begin{equation*}%\label{nash-cournot}
\text{s.t.} \quad \lr{F(x^*), x - x^*}\geq 0,\quad
\forall x \in \R^{n}_+,
\end{equation*}
where $F(x^*) = (F_1(x^*),\dots, F_n(x^*))$ and $ F_i(x^*) = f'_i(x_i^*) - p\Bigl(\sum_{j=1}^n x_j^*\Bigr) - x_i^*
p'\Bigl(\sum_{j=1}^n x_j^*\Bigr)$.\\
% \begin{equation*}
%  F_i(x^*) = f'_i(x_i^*) - p\Bigl(\sum_{j=1}^n x_j^*\Bigr) - x_i^*
% p'\Bigl(\sum_{j=1}^n x_j^*\Bigr).
% \end{equation*}
We assume that the function $p$ and $f_i$ are written as $p(Q) = 5000^{1/\gamma}Q^{-1/\gamma}$ and $f_i(x_i) = c_i x_i + \frac{\beta_i}{\beta_i+1}L_i^{\frac{1}{\beta_i}} x_i^{\frac{\beta_i+1}{\beta_i}}$.
%\begin{equation*}
 %p(Q) = 5000^{1/\gamma}Q^{-1/\gamma} \qquad \text{and} \qquad f_i(x_i) = c_i x_i + \frac{\beta_i}{\beta_i+1}L_i^{\frac{1}{\beta_i}} x_i^{\frac{\beta_i+1}{\beta_i}}
%\end{equation*}
We set $n=1000$ and generate our data randomly. Furthermore, we consider two scenarios for each entry of $\beta$, $c$, and $L$, which are drawn independently from the uniform distributions as follows:
% \begin{itemize}
\begin{enumerate}[label=(\roman*)]
    \item $\gamma =1.1$, $\beta_i\sim \mathcal{U}(0.5,2)$,
    $c_i\sim \mathcal{U}(1,100)$, $L_i\sim \mathcal{U}(0.5,5)$;
    \item $\gamma =1.5$, $\beta_i\sim \mathcal{U}(0.3,4)$ and $c_i$, $L_i$ as above.
    % \end{itemize}
\end{enumerate}
These parameters control the level of smoothness of $f_i$ and $p$; therefore, they can affect the convergence speed. Figure~\ref{fig2} reports the results where all algorithms are initialized at the same point chosen randomly: Our proposed algorithms exhibits faster convergence speed and outperforms other algorithms.
\begin{figure}[t]
    \centering
    \captionsetup{justification=centering}
    \subfloat[\footnotesize Case (i).]{\label{fig:VI1}\includegraphics[scale=0.23]{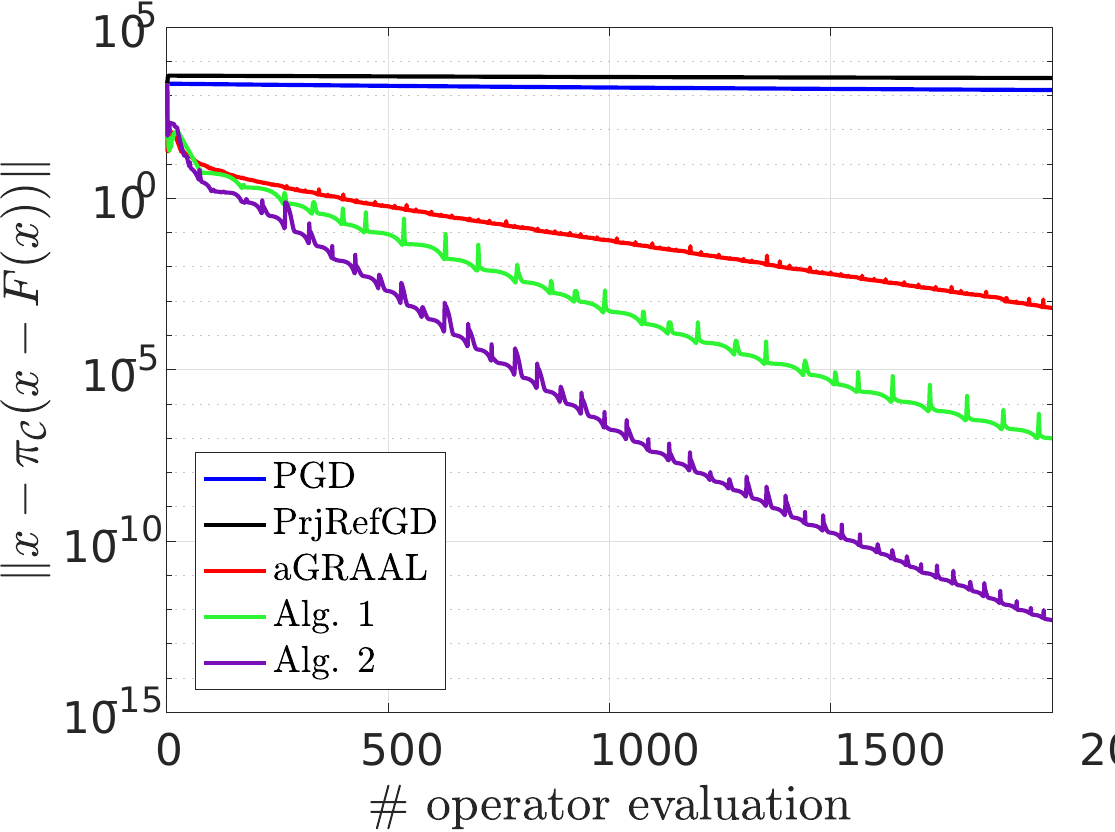}}
        %\hspace{.1em}
    \subfloat[\footnotesize Case (ii).]{\label{fig:VI2}\includegraphics[scale=0.23]{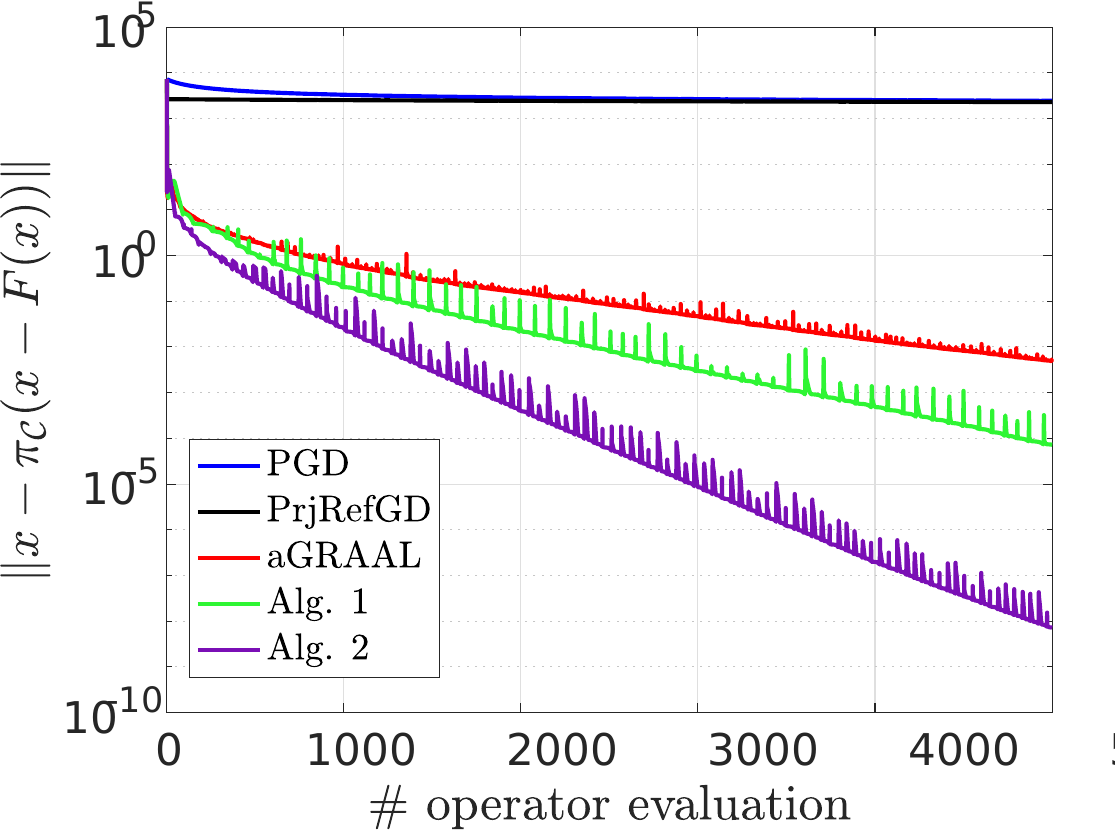}}
    \caption{\footnotesize Nash-Cournot equilibrium \ref{nash}.}
    \label{fig2}
\end{figure}
\item\textbf{Sparse logistic regression \cite{mishchenko2023regularized}. \label{sparse-l-r}} The sparse logistic regression can be written as follows
\begin{equation}
    \label{eq:slr}
     \min_x f(x):=\sum_{i=1}^m \log(1+\exp (-b_i\lr{a_i, x})) + \gamma \n{x}_1,
\end{equation}
where $x\in \R^n$, $a_i\in \R^n$, and $b_i\in \{-1, 1\}$, $\gamma>0$. This problem can be found in several machine learning applications, where one attempts to find a linear classifier for points $a_i$. The objective function in \eqref{eq:slr} is $f(x) = s(x)+g(x)$ with $g(x) = \gamma \n{x}_1$ and $s(x) =  h(Dx)$, where matrix $D\in \R^{m\times n}$ as $D_{ij} = -b_ia_{ij}$ and function $ h(y) = \sum_{i=1}^m\log(1+\exp (y_i))$.  It is easy to see that $s(x)$ is smooth with Lipschits constant gradient with $L_{\nabla s} = \frac 1 4 \n{D^{\top}D}$. In our experiments the test data $a_i$ and $b_i$ are generated randomly using the standard Gaussian distribution, $\gamma = 0.005\|A^\top b\|_{\infty}$, where $A = [a_1|a_2|\cdots|a_m]\in \mathbb{R}^{n\times m}$, $n = 500$, and $m = 200$. The results are presented in Figure \ref{fig:VI6}, where our methods demonstrates superior efficacy compared to other algorithms. We also plot the result of solving \eqref{eq:slr} using accelerated PrGD (FISTA) \cite{beck2009fast}.
\item\textbf{Two-player Zero Sum Game \cite{lemke1964equilibrium}. \label{2-zero-sum}} Generative adversarial networks (GANs) are a powerful class of neural networks that are used for unsupervised learning. The training of GANs can be considered a two-player zero sum game \cite{goodfellow2014generative}. For solving a two-player zero sum game, we need to solve the following bilinear saddle point problem,
\begin{equation}\label{eq:zero-sum}
	\min_{x \in \Delta^{m}} \, \max_{y \in \Delta^{n}} \; \Phi(x, y) := x^{\top} A y,
\end{equation}
where $ A \in \R^{m \times n} $ is a pay-off matrix and $ \Delta^{d} = \{ v \in \R^{d}_{+} \mid \sum_{i=1}^{d} v_{i} = 1 \} $ denotes the $ d $-dimensional simplex. The solution of~\eqref{eq:zero-sum} is given by a saddle point $ (x^{\ast}, y^{\ast}) \in \Delta^{m} \times \Delta^{n} $ satisfying $\Phi(x^{\ast}, y) \leq \Phi(x^{\ast}, y^{\ast}) \leq \Phi(x, y^{\ast})$
\begin{comment}
\begin{equation}
	\Phi(x^{\ast}, y) \leq \Phi(x^{\ast}, y^{\ast}) \leq \Phi(x, y^{\ast}),
	%\quad \forall (x,y) \in \Delta^{m} \times \Delta^{n}.
\end{equation}
\end{comment}
for all $(x,y) \in \Delta^{m} \times \Delta^{n}$, which can be written by problem~\eqref{VI-main} with $ \mathcal{A} = \Delta^{m} \times \Delta^{n} $ and
\begin{equation} \label{zerosumgame}
	F(x,y)
	=
	\begin{pmatrix}
		Ay\\
		\scalebox{0.75}[1.0]{$-$} A^{\top}x
	\end{pmatrix}
	=
	\begin{pmatrix}
		0 & A\\
		-A^{\top} & 0
	\end{pmatrix}
	\begin{pmatrix}
		x\\
		y
	\end{pmatrix}.
\end{equation}
For the experiments, we set $d = m = n = 50$, $A$ is generated with a uniform distribution on $\left[0, 1\right)$. A comparison of methods is reported in Figure~\ref{figGame2player}.
\vspace{-.2cm}
%%%%%%%%%%%%%%%%%%%%%%%%%%%%%
\begin{figure}[!h]
%\centering
    % \begin{minipage}{0.25\textwidth}
    % \centering
    % \captionsetup{justification=centering}
    % {\includegraphics[scale=0.21]{images/VIFeasibility.eps}}
    % \caption{\footnotesize Feasibility problem~\ref{feasibility}.}
    % \label{fig:VI5}
    % \end{minipage}
    \begin{minipage}{0.22\textwidth}
    \centering
     \captionsetup{justification=centering}
    {\includegraphics[scale=0.23]{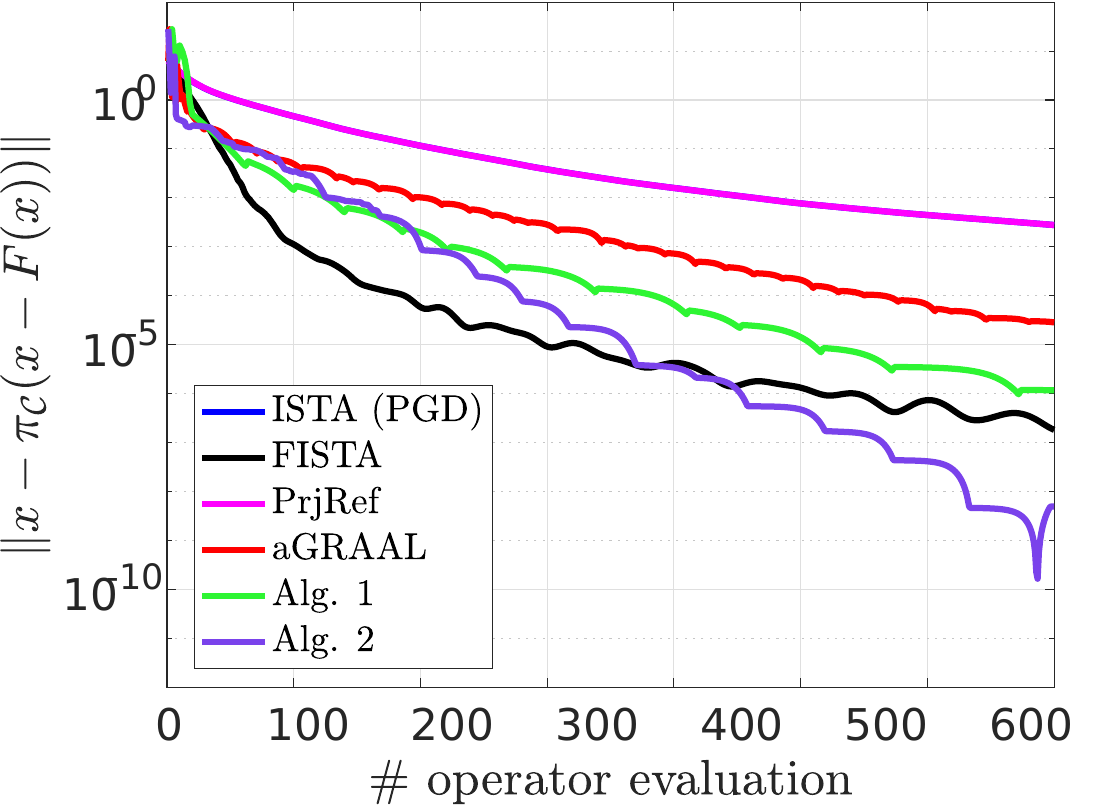}}
    \caption{\footnotesize Logistic regression~\ref{sparse-l-r}.}
    \label{fig:VI6}
    \end{minipage} 
    \hspace{1mm}
     \begin{minipage}{0.22\textwidth}
    \centering
     \captionsetup{justification=centering}
    {\includegraphics[scale=0.23]{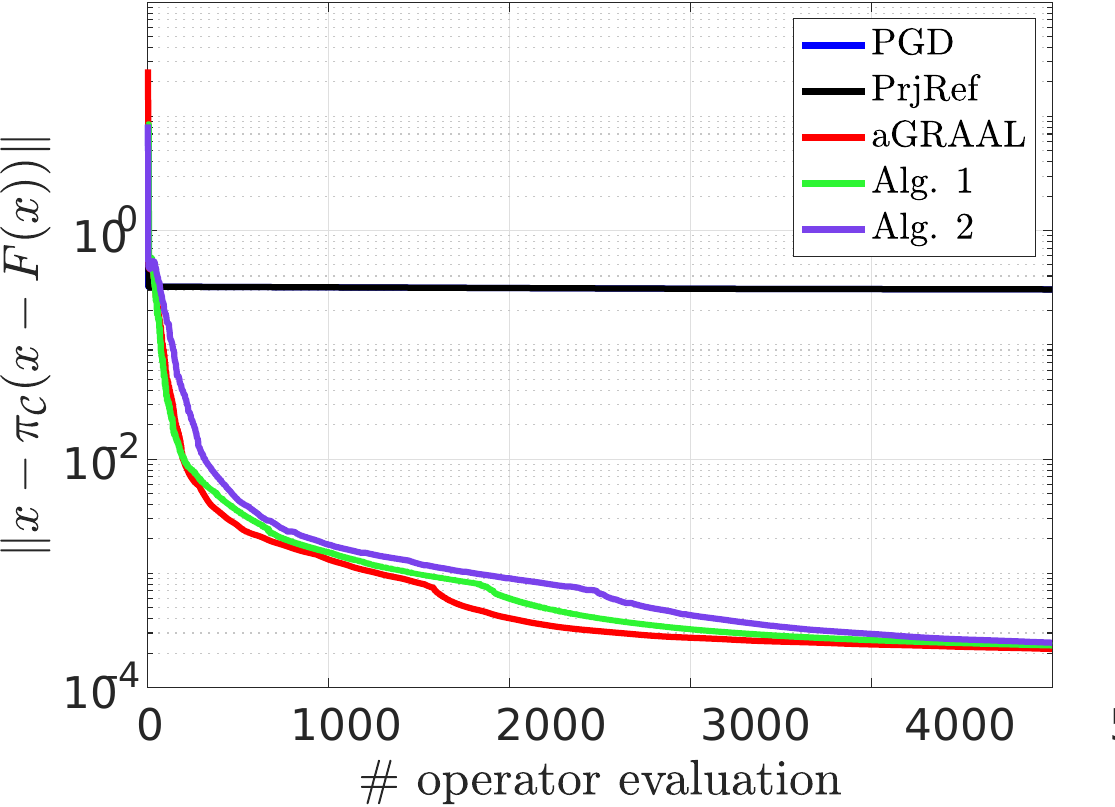}}
    \caption{\footnotesize Zero Sum Game \ref{2-zero-sum}.}
    \label{figGame2player}
    \end{minipage}
\end{figure}
% \begin{figure}[!h]
    %     \begin{minipage}{0.25\textwidth}
    %     \centering
    %     \captionsetup{justification=centering}
    %     {\includegraphics[scale=0.22]{images/VISkewsymmetric.eps}}
    %     \caption{\footnotesize Skew symmetric operator \ref{skew-sym}.}
    %     \label{fig:VI7}
    % \end{minipage}
%     \hspace{-.1cm}
%     \begin{minipage}{0.22\textwidth}
%     \centering
%      \captionsetup{justification=centering}
%     {\includegraphics[scale=0.22]{images/VItwoPlayer.eps}}
%     \caption{\footnotesize Zero Sum Game \ref{2-zero-sum}.}
%     \label{figGame2player}
%     \end{minipage}
% \end{figure}
%%%%%%%%%%%%%%%%%%%%%%%%%%%%
\vspace{-.2cm}
\item\textbf{Markov decision processes (MDPs) \cite{kolarijani2023optimization}. \label{mdp pr}} An MDP is a pair of $(\set{S},\set{A},\PP,c,\gamma)$, where $\set{S}$ and $\set{A}$ are the state space and action space, respectively. The transition kernel $\PP$ describes how the system moves between states: given a state $s$ and an action $a$, it shows the probability of transitioning to another state $s^+$.
The cost function $c: \set{S}\times\set{A} \rightarrow \mathbb{R}$, bounded from below, assigns a cost to each action-state pair. 
The discount factor $\gamma \in (0,1)$ can be seen as a trade-off parameter between short- and long-term costs. We take $\set{S} = \{1,2,\ldots,n \}$ and $\set{A} = \{1,2,\ldots,m \}$.\\
MDPs provide a robust modeling framework for stochastic environments, offering control mechanisms to minimize cost measures. By accessing to the transition kernel and the cost function, the problem is usually characterized by the fixed-point problem $ v^* = T(v^*)$, i.e., 
\begin{align}\label{eq:fp_v}
     v^*(s) = [T (v^*)](s), \quad \forall s \in \set{S},
\end{align}
where $T:\R^{|\set{S}|}\rightarrow\mathbb{R}^{|\set{S}|}$ is the \emph{Bellman operator} given by 
\begin{align*} %\label{eq:T}
    [T (v)](s) = \min_{a \in\set{A}} \left\{ c(s,a) + \gamma \mathds{E}_{s^+\sim\PP(\cdot|s,a)} \left[v(s^+)\right] \right\}. 
\end{align*}
The optimal value function~$v^*$ is the unique fixed-point of the Bellman operator~$T$. Therefore, we can reformulate this problem with \eqref{VI-main} by $F = \text{Id} - T$ and $g(x) = 0$.
%The uniqueness follows from the fact that the operator $T$ is a $\gamma$-contraction in $\infty$-norm. 
The comparison of the proposed algorithms in solving 50 instances of the optimal control problems of randomly generated Garnet MDPs with $n = 50$ states and $m = 5$ actions with two different values of discount factor $\gamma$ is reported in Figure \ref{figMDP}.
\vspace{-.5cm}
%%%%%%%%%%%%%%%%%%%%%%%%%%%%%%%%%%%%%%%%%%%%%
\begin{figure}[!h]
    \centering
    \captionsetup{justification=centering}
    \subfloat[$\gamma = 0.9$.]{\label{fig:VI9}\includegraphics[scale=0.23]{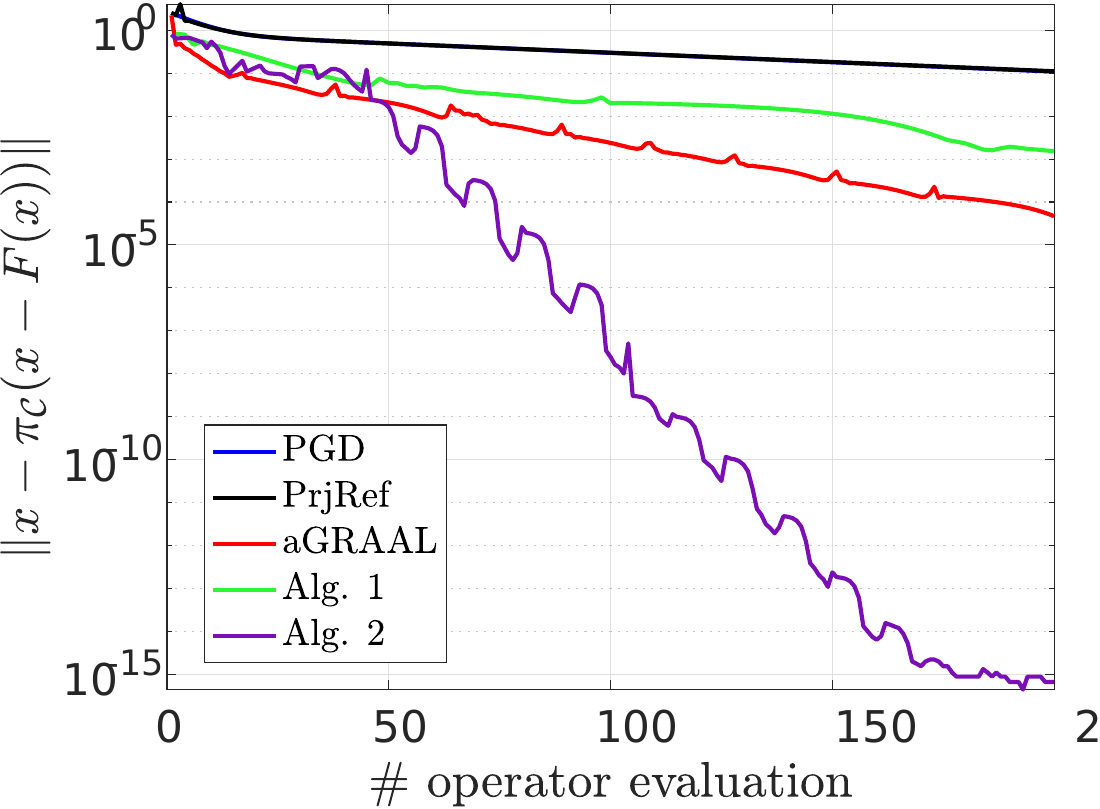}}
        \hfill
    \subfloat[$\gamma = 0.99$.]{\label{fig:VI10}\includegraphics[scale=0.23]{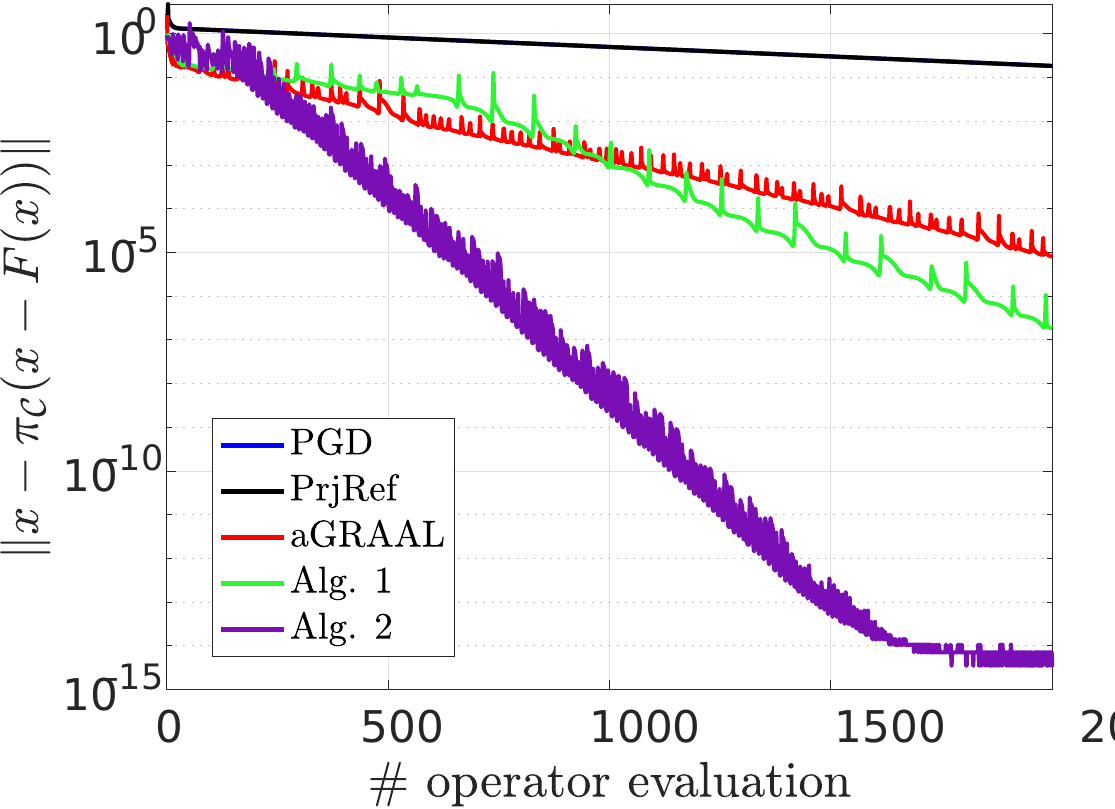}}
    %\hspace{.5em}
    %\subfloat[$\gamma = 0.999$.]{\label{fig:VI11}\includegraphics[scale=0.20]{images/VIMDP0.999.eps}}
    \caption{\footnotesize Performance in MDP for different values of $\gamma$ \ref{mdp pr}.}
    \label{figMDP}
\end{figure}
%%%%%%%%%%%%%%%%%%%%%%%%%%%%%%%%%%%%%%%%%%%%%
\vspace{-.25cm}
\item\textbf{Strongly monotone affine operator \cite{baghbadorani2025douglas}. \label{s-monotone-vi}} One popular VI problem with a strongly monotone operator is \eqref{VI-main} with affine operator $F(x) = Mx + q$, where $M$ generated randomly as $M = AA^\top + B + D,$ where each entry of the $n \times n$ matrix $A$ and the skew-symmetric matrix $B$ is uniformly sampled from the interval $(-5, 5)$, and every entry of diagonal matrix $D$ is uniformly sampled from the interval $(0, 0.3)$ (ensuring $M$ is positive definite), with each entry of $q$ uniformly sampled from $(-500, 0)$. The feasible set is $\mathcal{A} = \{ x \in \R^{n}_+ |\, x^1 + x^2 + \cdots + x^n = n\}$. For simulation experiments, we consider $n = 100$ and $L = |M|$ as the Lipschitz continuity of $F$. Figure~\ref{fig:VIstrongmonotone} illustrates the results with initial point $x^0 = (1,1,\ldots,1)$.
%\begin{comment}
\item\textbf{Non-monotone operator. \label{nonmonotone vi}} As a last example, we test our proposed algorithms on a non-monotone operator mentioned in \cite{malitsky2020golden}, where we aim to find a non-zero solution of $F(x) := M(x)x = 0$. Here, $M\colon \mathbb{R}^n\to \mathbb{R}^{n\times n}$ is a matrix-valued function, which can be considered as a VI \eqref{VI-main} with $g = 0$. For the experiment, we define $M$ as $M(x) := t_1t_1^{\top} + t_2t_2^{\top}, \, \text{with}\, t_1 = A\sin x, \, t_2 = B \exp(x),$
%\begin{equation*}
%    \label{eq:numF}
%     M(x) := t_1t_1^{\top} + t_2t_2^{\top}, \, \text{with}\, t_1 = A\sin
%    x, \, t_2 = B \exp(x),
%\end{equation*}
where $x\in \mathbb{R}^n$, $A$ and $B \in \mathbb{R}^{n\times n}$. For the experiment, we choose $n = 500$, and the matrices $A$ and $B$ are independently and randomly generated from the normal distribution $\mathcal{N}(0,1)$. The results of solving VI with the non-monotone operator $M$ are reported in Figure \ref{fig:VI8}, where the proposed algorithms outperform other methods.
%We refer interested readers to \cite{mignoni2025monviso} for additional algorithms for applications of monotone variational inequalities, as well as a ready-to-use Python toolbox.
%\end{comment}
\begin{figure}[!h]
    \begin{minipage}{0.22\textwidth}
        \centering
        \captionsetup{justification=centering}
        {\includegraphics[scale=0.23]{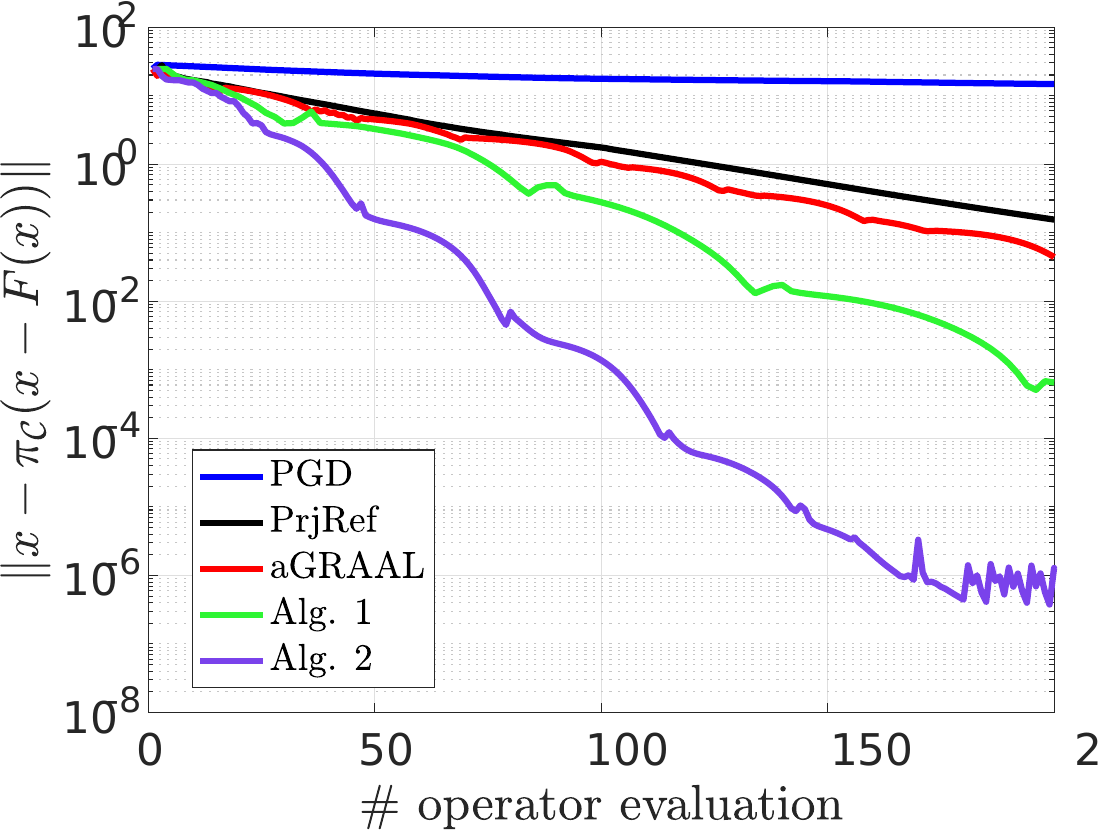}}
        \caption{\footnotesize Strongly monotone operator~\ref{s-monotone-vi}.}
        \label{fig:VIstrongmonotone}
    \end{minipage}
    \hspace{.1cm}
    \begin{minipage}{0.22\textwidth}
    \centering
    \captionsetup{justification=centering}
   {\includegraphics[scale=0.23]{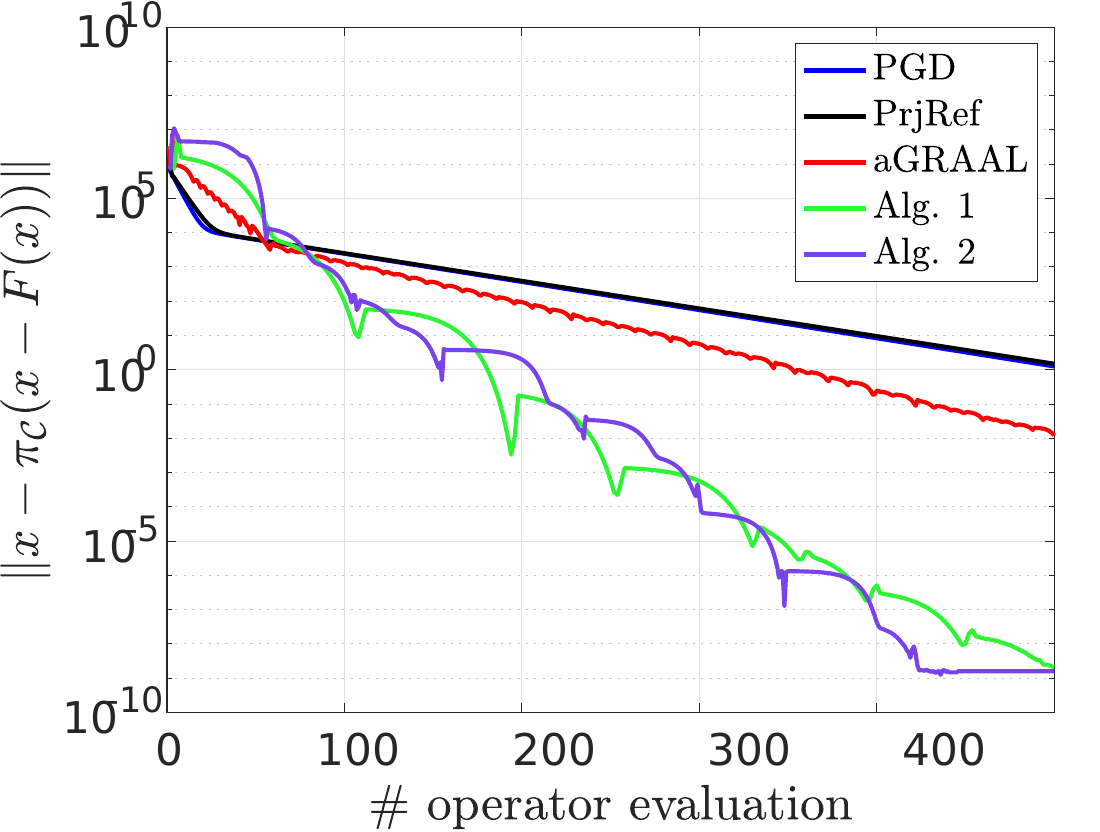}}
   \caption{\footnotesize Non-monotone operator~\ref{nonmonotone vi}.}
   \label{fig:VI8}
    \end{minipage}
\end{figure}
\end{enumerate}
%%%%%%%%%%%%%%%%%%%%%%%%%%%%%%%%%%%%%%%%
\section{Appendix}
In this section, we provide the proof of the Theorem \ref{Ergodic convergence} and additional supporting material.
\begin{definition}[Cluster point]
A point $x$ is called a cluster point of the sequence $\{x^k\}$ if there exists a subsequence $\{x^{k_j}\}$ such that $\lim_{j \to \infty} x^{k_j} = x$.
\end{definition}
\begin{lemma}[Bolzano–Weierstrass theorem]\label{lemma-cluster}
      If ${x^k} \in \mathcal{V}$ is a bounded sequence, and $\lim_{k\rightarrow \infty} (x^k - x)$ exists, where $x$ is a cluster point of the sequence ${x^k}$, then $x^k$ is convergent.
\end{lemma}
\begin{customproof}[Proof of Theorem \ref{Ergodic convergence}]
We use same proof technique as in \cite{malitsky2020golden}. If $ x = x^* \in \mathcal{S} $ in \eqref{conv ineq}, where $ \mathcal{S} $ is the solution set of \eqref{VI-main}, the sequences $ {x^k} $ and $ {\bar{x}^k} $ are bounded, with $ \theta_k |x^k - \bar{x}^k| \to 0 $. By \cite[Lemma 2]{malitsky2020golden}, $ \lambda_k $ and $ \theta_k $ remain bounded away from zero, implying $ x^k - \bar{x}^{k-1} \to 0 $ and $ x^{k+1} - x^k \to 0 $. For any cluster point of $ (x^k) $ and $ (\bar{x}^k) $, let $ (k_i) $ be a subsequence such that $ x^{k_i} \to \hat{x} $ and $ \lambda_{k_i} \to \lambda > 0 $. Then, $ x^{k_i+1} \to \hat{x} $ and $ \bar{x}^{k_i} \to \hat{x} $. Taking the limit of \eqref{ls:eq:4} for this subsequence gives: $\lambda \langle F(\hat{x}), x - \hat{x} \rangle \geq \lambda (g(\hat{x}) - g(x)), \quad \forall x \in \mathcal{V}$.
%\begin{align*}
%\lambda \langle F(\hat{x}), x - \hat{x} \rangle \geq \lambda (g(\hat{x}) - g(x)), \quad \forall x \in \mathcal{V}.
%\end{align*}
Thus, $ \hat{x} \in \mathcal{S} $. Finally, by Lemma~\ref{lemma-cluster}, the sequence $ (x^k) $ converges to some point in $\mathcal{S}$. Now, by defining a merit function $e_r(y) := \max_{x \in \mathcal{U}} \, \Psi(x,y)$ for all $y \in \mathcal{V}$, where $\mathcal{U} = \dom \, g \cap \mathbb{B}(\hat{x},r)$ and $\hat{x} \in \dom g$, as shown in \cite[Lemma~3]{malitsky2020golden}, we know that $e_r(y)$ is a positive convex function for all $y \in \mathcal{U}$. Furthermore, if $e_r(\Tilde{x}) = 0$ for some $\Tilde{x}$ where $\|\Tilde{x} - \hat{x}\| \leq r$, then $\Tilde{x}$ is a solution of \eqref{VI-main}. Since we assume that $F$ is a continuous operator and $g$ is lsc, there exists a constant $M > 0$ that bounds the right-hand-side of \eqref{conv ineq} for all $x \in \mathcal{U}$. Consequently, $\sum_{i=1}^k \lambda_i \Psi(x, x^i)$ can be bounded above by the constant $M$ for all $x \in \mathcal{U}$. Finally, let $X^k$ be the ergodic sequence defined in Theorem \ref{Ergodic convergence}.
 Using the convexity of $ \Psi(x, \cdot) $, we obtain $e_r(X^k) = \max_{x \in U} \Psi(x, X^k) \leq \frac{\max_{x \in U} \left( \sum_{i=1}^k \Psi(x, x^i) \right)}{\sum_{i=1}^k \lambda_i} 
\leq \frac{M}{\sum_{i=1}^k \lambda_i}.$
%\begin{align*}
%e_r(X^k) &= \max_{x \in U} \Psi(x, X^k) &\leq \frac{\max_{x \in U} \left( \sum_{i=1}^k \Psi(x, x^i) \right)}{\sum_{i=1}^k \lambda_i} \\
%&\leq \frac{M}{\sum_{i=1}^k \lambda_i}. 
%\end{align*}
This implies an ergodic convergence rate of $\mathcal{O}(k^{-1})$ as $\lambda_k$ is separated from zero.
\end{customproof}
%%%%%%%%%%%%%%%%%%%%%%%%%%%%%%%%%%%%%%%%%
\vspace{-2mm}
\bibliographystyle{IEEEtran} %{siam}
\bibliography{references}

\end{document}